\documentclass{article}

\usepackage[all]{xy}
\usepackage{mathrsfs}
\usepackage{amssymb}
\usepackage{amsthm}
\usepackage{amsmath}

\usepackage{tikz}
\usetikzlibrary{arrows,positioning}

\newtheorem{thm}{Theorem}

\newtheorem{cor}[thm]{Corollary}

\newtheorem{defn}{Definition}
\newtheorem{exa}{Example}

\newtheorem{rem}{Remark}

\allowdisplaybreaks

\author{Camilo Sanabria Malag\'on\thanks{Partially supported by Vicerrector\'ia de
Investigaciones de la Universidad de los Andes grant PEP P13.160422.030 FAPA – Camilo Sanabria.}\\Department of Mathematics\\ Universidad de los Andes, Bogot\'a, Colombia}

\title{Schwarz maps of Algebraic Linear Ordinary Differential Equations}

\begin{document}

\date{}

\maketitle

\begin{abstract}
A linear ordinary differential equation is called algebraic if all its solution are algebraic over its field of definition. In this paper we solve the problem of finding closed form solution to algebraic linear ordinary differential equations in terms of standard equations. Furthermore, we obtain a method to compute all algebraic linear ordinary differential equations with rational coefficients by studying their associated Schwarz map through the Picard-Vessiot Theory.
\end{abstract}

\section*{Introduction}
 
Since the introduction of the concept of local systems by B. Riemann \cite{RIEMANN1953,RIEMANN1968}, the study of linear ordinary differential equations with rational coefficients has been dominated by his viewpoint. Among the most prominent works based on this perspective are P. Deligne's work on the Riemann-Hilbert correpondence and Hilbert's twenty-first problem \cite{DELIGNE1970}, M. Kashiwara's work on the Riemann-Hilbert problem for holonomic systems \cite{KASHIWARA1984}, and B.Malgrange's exposition on the Riemann-Hilbert-Birkhoff correspondence \cite{MALGRANGE1991}. Although this approach of Riemann is mainly used in the study of equations with regular singularities, i.e. fuchsian systems, introduced by L. Fuchs in \cite{FUCHS1909}, it also gave rise to different tools for studying the local behavior around irregular singularities. A modern perspective of this theory was introduced by N. M. Katz through his concept of rigid local systems \cite{KATZ1996}.

Another perspective is given by Picard-Vessiot theory, developed by E. Picard \cite{PICARD1991III}, E. Vessiot \cite{VESSIOT1892} and E. R. Kolchin \cite{KOLCHIN1946}. With this theory, a differential Galois group is associated to a linear ordinary differential equation. It was proved by J.-P. Ramis \cite{RAMIS1985,RAMIS1990} and J. J. Morales \cite{MORALES1999,MORALES2001} that this group is obtained by combining the monodromy, exponential tori and Stoke's matrices of the local system in a linear algebraic group . The galoisian approach has proved successful at obtaining closed form solutions to the equation. Indeed E. R. Kolchin proved in \cite{KOLCHIN1946} that the solutions are liouvillian if and only if the differential Galois group is solvable, and for the second order case this result has been systematized by J. Kovacic's algorithm \cite{KOVACIC1986}. M. van Hoeij, J. F. Ragot, M.Singer, F. Ulmer and J.-A. Weil have extended this algorithm to higher order equations \cite{VANHOEIJ1999,SINGER1993II,SINGER1993,ULMER2003}. Using those algorithms the solutions are presented as the exponential of the antiderivative of a function which is algebraic over the field of definition of the equation. This algebraic function is described in terms of its minimal polynomial, which is computed using the (semi-)invariants of the differential Galois group. Therefore, the smaller the group the more computations are involved, reaching a maximum when the differential Galois group is finite. This is equivalent to the case when all the solutions to the linear ordinary differential equation are algebraic over its field of definition. These equations are called algebraic linear ordinary differential equations. 

A classical result of F. Klein \cite{KLEIN187711,KLEIN187712} states that a full system of solutions to an algebraic second order linear ordinary differential equation with rational coefficients can be given in closed form as solutions to hypergeometric equations precomposed with a rational function, multiplied by the solution to a first order linear differential equation. The first modern treatment of Klein's result was done by B. Dwork and F. Baldassari in \cite{BALDASSARRI1980,BALDASSARRI1979}. To overcome the difficulty of the extensive computations required in the algebraic case in Kovacic's algorithm, M. Berkenbosch applied in \cite{BERKENBOSCH2006} this result of Klein and he generalized it to order three by introducing the concept of standard equations. Furthermore, C. Sanabria in \cite{SANABRIA2014} extended it to arbitrary order. Therefore, the problem of finding closed form solutions to algebraic linear ordinary differential equations has been reduced to the problem of first, characterizing the family of standard equations, and second, identifying which standard equation is needed in order to obtain closed form solutions to a given equation. Both of these problems have been solved in this paper. Furthermore, we obtain a method to compute explicitly all irreducible algebraic linear ordinary differential equations.

The main tool used in the extensions of Klein's theorem in \cite{BERKENBOSCH2006,SANABRIA2014} are Schwarz maps. Schwarz maps offer a third approach to the study of linear ordinary differential equations which was used by H. A. Schwarz \cite{SCHWARZ1873}, F. Klein \cite{KLEIN1878,KLEIN1879} and G. Fano \cite{FANO1900}. Their most famous application is the description of tessellations of the Riemann sphere by Schwarz triangles \cite[Sections III.6,III.7,III.8]{YOSHIDA1997}. With Schwarz approach, instead of studying local systems, a full system of solutions of a linear differential equation is understood as a parametrization of a complex curve in a complex projective space.

In this paper, we study Schwarz maps from the perspective of the Picard-Vessiot theory using Compoint's theorem \cite{COMPOINT1998}. In the first section, we construct a method to identify the standard equation needed to obtain a closed form solutions to a given algebraic linear ordinary differential equation. We also characterize standard equations by showing that, up to projective equivalence, their associated Schwarz map corresponds to an orbit of a differentiable dynamical system (Theorem \ref{thmI}). On the other hand, in the second section, we show a family of differential dynamical system whose orbits are Schwarz maps of standard equations (Theorem \ref{thmII}). Throughout the paper we use P\'epin's equation \cite{PEPIN1881} and Hurwitz's equation \cite{HURWITZ1886} to illustrate the theory. Moreover, in the third section, we apply this approach to the classical Hesse pencil \cite{HESSE1844I,HESSE1844II}, which is a one-parameter continuous family of elliptic curves that has gain much attention recently in different areas \cite{ARTEBANI2009,SMART2001,ZASLOW2005}. We obtain a one-parameter family of standard equations whose associated Schwarz maps parametrize this family of elliptic curves. We also obtain an equation that explains how these parametrizations deform as we vary the parameter of the Hesse pencil. Additionally, in the fourth section, as another application of our approach, we study the family of standard equations whose associated Schwarz maps parametrize the curves in the Fricke degree $12$ pencil \cite{FRICKE1893,KATO2004}. We also show that in this continuous family of equations all the generic cases are deformations of each other, and the singular cases correspond to singularities of an equation in terms of the deformation parameter. In particular, all the curves in the Fricke degree $12$ pencil are actually orbits of a common differentiable dynamical system.

All the examples in this paper were carried out using MAPLE and the codes can be obtained from my webpage.

\section{Schwarz maps and Picard-Vessiot Theory for Linear Ordinary Differential Equations}

Consider an irreducible linear ordinary differential equation $L(x)=0$, where
\begin{equation*}\label{LODE}
L(x)=\left(\frac{d}{dt}\right)^n x+a_{n-1}\left(\frac{d}{dt}\right)^{n-1}x+\ldots+a_1\frac{d}{dt}x+a_0x,
\end{equation*}
$x=x(t)$ and $a_{n-1},\ldots,a_1,a_0\in\mathbb{C}(t)$.

A Picard-Vessiot extension for $L(x)=0$ is the field
$$K=\mathbb{C}(t)\left(\left(\frac{d}{dt}\right)^{i-1} x_j \Big|\ i,j=1,\ldots,n\right),$$
where $x_j=x_j(t)$, for $j=1,\ldots,n$ form a full system of solutions \cite[Proposition 1.22]{VANDERPUT2003}.

Let $J$ be the kernel of the $\mathbb{C}(t)$-morphism
\begin{eqnarray*}
\Psi: \mathbb{C}(t)\Big[X_{ij}|\ i,j=1,\ldots,n\Big]\left[\frac{1}{\det(X_{ij})}\right] & \longrightarrow & K\\
X_{ij} & \longmapsto & \frac{d^{i-1}}{dt^{i-1}}x_j.
\end{eqnarray*}
An isomorphic representation of the differential Galois group of $L(x)=0$ is the algebraic group $G\subset GL_n(\mathbb{C})$ composed by the elements $(g_{ij})_{i,j=1}^n$ sending the ideal $J$ into itself under the $\mathbb{C}(t)$-automorphisms \cite[Observations 1.26]{VANDERPUT2003}
\begin{eqnarray*}
\mathbb{C}(t)\Big[X_{ij}|\ i,j=1,\ldots,n\Big]\left[\frac{1}{\det(X_{ij})}\right]  & \longrightarrow & \mathbb{C}(t)\Big[X_{ij}|\ i,j=1,\ldots,n\Big]\left[\frac{1}{\det(X_{ij})}\right] \\
X_{ij} & \longmapsto & \sum_{l=1}^nX_{il}g_{lj}.
\end{eqnarray*}
By the Galois correspondence \cite[Proposition 1.34]{VANDERPUT2003} if $P\in\mathbb{C}(t)[X_{ij}|\ i,j=1,\ldots,n]$ is $G$-invariant, then $\Psi(P)\in \mathbb{C}(t)$. Since $L(x)=0$ is irreducible, hence $G$ is reductive \cite[Exercise 2.38]{VANDERPUT2003}. 

Under these assumptions, we have:
\begin{thm}[Compoint's theorem \cite{BEUKERS2000}]\label{Compointthm} If the differential Galois group of $L(x)=0$ is reductive, the ideal $J$ is generated by the $G$-invariants contained in it.
\end{thm}

\subsection{Picard-Vessiot Theory in the algebraic case}
From now on we will assume that all the solutions to $L(x)=0$ are algebraic over $\mathbb{C}(t)$. In particular $G$ is finite and the Picard-Vessiot extension is
$$K=\mathbb{C}(t)\left[x_1,\ldots,x_n\right].$$

Let $I\subset\mathbb{C}(t)[X_1,\ldots,X_n]$ be the kernel of the $\mathbb{C}(t)$-morphism
\begin{eqnarray*}
\Phi: \mathbb{C}(t)[X_1,\ldots,X_n] & \longrightarrow & K\\
X_j & \longmapsto & x_j
\end{eqnarray*}
As above, the Galois correspondence implies that if $P\in\mathbb{C}(t)[X_1,\ldots,X_n]$ is $G$-invariant, then $\Phi(P)\in \mathbb{C}(t)$. We have

\begin{thm}[Compoint's theorem for the algebraic case] Let $L(x)=0$ be an algebraic irreducible linear ordinary differential equation. If the differential Galois group of $L(x)=0$ is finite, the ideal $I$ is generated by the $G$-invariants contained in it.
\end{thm}

\begin{proof} From Theorem \ref{Compointthm}  we have that $J$ is generated by the $G$-invariants it contains. Since $\mathbb{C}(t)[X_{11},\ldots,X_{1n}]$ is closed under the action of $G$, the contraction of $J$ to $\mathbb{C}(t)[X_{11},\ldots,X_{1n}]$ is also  generated by the $G$-invariants it contains. The claim follows now by noting that $I$ corresponds to $J\cap\mathbb{C}(t)[X_{11},\ldots,X_{1n}]$ where $X_{1j}:=X_j$, for $j=1,\ldots,n$.
\end{proof}

\begin{cor}
If $P_1,\ldots,P_N\in\mathbb{C}[X_1,\ldots,X_n]$ is a set of generators of the $G$-invariant subring $\mathbb{C}[X_1,\ldots,X_n]^G$, then
\begin{equation}\label{Igen}
I=\langle P_1-f_1,\ldots, P_N-f_N\rangle,
\end{equation}
where $f_i=\Phi(P_i)$, $i=1,\ldots,N$.
\end{cor}

\subsection{A first order differential equation}\label{1stDE}

In this sections, starting from the solutions $x_1,\ldots,x_n$ to $L(x)=0$, we will construct a non-autonomous first order differential equation with solution $$\mathbf{X}(t)=(x_1,\ldots,x_n).$$

Differentiating the generators of $I$ in (\ref{Igen}) we obtain the relations
\begin{equation}\label{difgen}
\sum_{j=1}^n\frac{\partial P_i}{\partial X_j}(x_1,\ldots,x_n)\frac{d}{dt}x_j=\frac{d}{dt}f_i,\quad i=1,\ldots,N.
\end{equation}
Since $G$ is finite, the orbit space $\mathbb{C}^n/G$ has dimension $n$, therefore the derivative of the map of projection into the orbits
\begin{eqnarray*}
\Pi_G: \mathbb{C}^n & \longrightarrow & \mathbb{C}^n/G\subseteq\mathbb{C}^N\\
\mathbf{X}=(X_1,\ldots,X_n) & \longmapsto & \left(P_1(\mathbf{X}),\ldots,P_N(\mathbf{X})\right)
\end{eqnarray*}
is non-singular in a dense subset $U\subseteq\mathbb{C}^n$. In particular, after arranging the indices of the $P_i$'s if necessary, we may assume that
$$M(\mathbf{X})=\left[\begin{array}{ccc}
\partial P_1/\partial X_1 & \ldots & \partial P_1/\partial X_n\\
\vdots & \ddots & \vdots\\
\partial P_n/\partial X_1 & \ldots & \partial P_n/\partial X_n
\end{array}\right](\mathbf{X})$$
is invertible over $U$.

Using $n$ first equations of the system (\ref{difgen}) we define a $G$-invariant first order differential equation over $U$
\begin{equation}\label{difeq}
\frac{d}{dt}\mathbf{X}=F(t,\mathbf{X})
\end{equation}
where
\begin{equation}\label{Fdifeq}
F(t,\mathbf{X})=\big[M(\mathbf{X})\big]^{-1}\left[\begin{array}{c} \frac{d}{dt}f_1\\ \vdots \\ \frac{d}{dt}f_n\end{array}\right].
\end{equation}

We illustrate this construction in the following example.

\begin{exa}\label{PepinI}
Let us consider P\'epin's equation
$$\left(\frac{d}{dt}\right)^2x+\frac{21}{100}\frac{t^2-t+1}{t^2(t-1)^2}x=0$$
with differential Galois group $A_5^{SL_2}$ \cite{SINGER1993II}. Using the algorithm by M. van Hoeij and J.-A. Weil in \cite{VANHOEIJ1997} we compute the generators of the invariants in $\mathbb{C}[X_1,X_2]$. Starting from the series expansion of two linearly independent solutions around $t=0$,  we obtain
\begin{eqnarray*}
P_1(X_1,X_2) & = & X_1^{11}X_2-\frac{11}{256}X_1^6X_2^6-\frac{1}{65536}X_1X_2^{11} \\
P_2(X_1,X_2) & = & \frac{144027}{2097152}X_1^{20}+\frac{57}{64}X_1^{15}X_2^5+\frac{247}{32768}X_1^{10}X_2^{10}-\frac{57}{4194304}X_1^5X_2^{15}\\
& & \quad +\frac{1}{4294967296}X_2^{20}
\end{eqnarray*}
with their respective evaluations
\begin{eqnarray*}
P_1(x_1,x_2) & = & t^4(t-1)^4\\
P_2(x_1,x_2) & = & t^6(t-1)^6(t^2-t+1).
\end{eqnarray*}
So $\mathbf{X}(t)=(x_1,x_2)$ is a solution to the system of first order linear differential equations
{\small
\begin{eqnarray*}
\frac{d}{dt}X_1 & = & -65536X_2^4(956301312X_1^{15}+16187392X_1^{10}X_2^5-43776X_1^5X_2^{10}+X_2^{15})\\
 & &\quad (4t^3(t-1)^3(2t-1))\frac{1}{\Delta(X_1,X_2)}\\
 & &\ +\frac{1073741824}{5}X_1(65536X_1^{10}-16896X_1^5X_2^5-11X_2^{10})\\
 & &\quad (t^5(t-1)^5(2t-1)(7t^2-7t+6))\frac{1}{\Delta(X_1,X_2)}\\
\frac{d}{dt}X_2 & = & 16777216X_1^4(1152216X_1^{15}+11206656X_1^{10}X_2^5+63232X_1^5X_2^{10}-57X_2^{15})\\
  & &\quad (4t^3(t-1)^3(2t-1))\frac{1}{\Delta(X_1,X_2)}\\
  & &\ -\frac{1073741824}{5}X_2(720896X_1^{10}-16896X_1^5X_2^5-X_2^{10})\\
  & &\quad (t^5(t-1)^5(2t-1)(7t^2-7t+6))\frac{1}{\Delta(X_1,X_2)}
\end{eqnarray*} 
}
where
{\small
\begin{eqnarray*}
\Delta(X_1,X_2)& = & 19330976710656X_1^{30}\\
   & &\ -506361069699072X_1^{25}X_2^5-42927147796480X_1^{20}X_2^{10}\\
   & &\ \ -655687680X_1^{10}X_2^{20}+133632X_1^5X_2^{25}+X_2^{30}.
\end{eqnarray*}
}
\end{exa}

Note that $F(t,\mathbf{X})$ in (\ref{Fdifeq}) depends upon the choice of solutions $x_1,\ldots,x_n$ and the invariants $P_1,\ldots,P_n$.

\subsection{Schwarz maps, projective equivalence and pullbacks}

The Schwarz map \cite{YOSHIDA1997} of $L(x)=0$ associated to the solutions $x_1,\ldots,x_n$ is the analytic extension of the parametric curve in $\mathbb{P}^{n-1}(\mathbb{C})$
$$\left[\mathbf{X}\right](t)=(x_1:\ldots:x_n).$$ 
The image of the the Schwarz map is called the Fano curve \cite{SANABRIA2014}.

\begin{rem}
Since the solutions $x_1,\ldots,x_n$ are algebraic over $\mathbb{C}(t)$, the degree of transcendence of $\mathbb{C}(t)[x_1,\ldots,x_n]$ over $\mathbb{C}(t)$ is zero. Therefore the degree of transcendence of $\mathbb{C}[x_1,\ldots,x_n]$ over $\mathbb{C}$ is one. Thus, the algebraic closure of the analytic curve parametrized by $(x_1,\ldots,x_n)$ in $\mathbb{C}^n$ has dimension one, and so does its projection into $\mathbb{P}^{n-1}(\mathbb{C})$, which is the Fano curve. 
\end{rem}

Note that the Schwarz map and the Fano curve are uniquely determined up to a rational transformation of $\mathbb{P}^{n-1}(\mathbb{C})$.

\begin{defn}
Let $L(x)=0$ be a linear ordinary differential equation, where
\begin{equation*}
L(x)=\left(\frac{d}{dt}\right)^n x+a_{n-1}\left(\frac{d}{dt}\right)^{n-1}x+\ldots+a_1\frac{d}{dt}x+a_0x,
\end{equation*}
$x=x(t)$ and $a_{n-1},\ldots,a_1,a_0\in\mathbb{C}(t)$ and let $L_1(y)=0$ be a linear ordinary differential equation, where
\begin{equation*}\label{LODE1}
L_1(y)=\left(\frac{d}{dt}\right)^n y+b_{n-1}\left(\frac{d}{dt}\right)^{n-1}y+\ldots+b_1\frac{d}{dt}y+b_0y,
\end{equation*}
$y=y(t)$ and $b_{n-1},\ldots,b_1,b_0\in\mathbb{C}(t)$. We say that $L_1(y)=0$ is rationally projectively equivalent to $L(x)=0$ if there exists a function $f$ such that:
\begin{itemize}
\item $\frac{1}{f}\frac{d}{dt}f\in\mathbb{C}(t)$
\item every solution to $L_1(y)=0$ is of the form $y_0(t)=f(t)x_0(t)$ where $x_0$ is a solution to $L(x)=0$.
\end{itemize}
\end{defn}

In particular, when $L_1(y)$ is projectively equivalent to $L(x)=0$, the Schwarz map of $L_1(y)$ associated to the solutions $fx_1,\ldots,fx_n$ coincides with the Schwarz map of $L(x)=0$ associated to $x_1,\ldots,x_n$. Therefore post-composing these two Schwarz maps with the projection into the projective orbits
\begin{eqnarray*}
\Pi_{G,\mathbb{P}}: \mathbb{P}^{n-1}(\mathbb{C}) & \longrightarrow & \mathbb{P}(\mathbb{C}^{n}/G),
\end{eqnarray*}
will result in the same image.

We illustrate the concept of projective equivalence with an example.

\begin{exa}\label{PepinII}
We continue with P\'epin's equation from Example \ref{PepinI} keeping the same notation. As homogeneous coordinates of the space of projective orbits we take the two invariants of degree $60$, $P_1^5$ and $P_2^3$. In this homogeneous coordinate system
\begin{eqnarray*}
\Pi_{G,\mathbb{P}}(x_1,x_2) & = & (P_1(x_1,x_2)^5:P_2(x_1,x_2)^3)\\
 & = & (t^{20}(t-1)^{20}\ :\ t^{18}(t-1)^{18}(t^2-t+1)^3)\\
 & = & (1\ :\ \frac{(t^2-t+1)^3}{t^2(t-1)^2}).
\end{eqnarray*}
Now, let us consider the projectively equivalent equation
\begin{equation}\label{Pepin2}
\left(\frac{d}{dt}\right)^2y+\frac{2}{3}\frac{2t-1}{t(t-1)}\frac{d}{dt}y-\frac{11}{900}\frac{t^2-t+1}{t^2(t-1)^2}y=0.
\end{equation}
Its solutions are of the form
$$ y_0(t)=\frac{1}{t^{1/3}(t-1)^{1/3}}x_0(t),$$
where $x_0$ is a solution to $L(x)=0$ and we have
\begin{eqnarray*}
P_1( y_1, y_2 )^5 & = & 1\\
P_2( y_1, y_2 )^3 & = & \frac{(t^2-t+1)^3}{t^2(t-1)^2},
\end{eqnarray*}
where
$$ y_i(t)=\frac{1}{t^{1/3}(t-1)^{1/3}}x_i(t),\quad i=1,2.$$
Note that if we change the parameter $t$ by
$$s=\frac{(t^2-t+1)^3}{t^2(t-1)^2},$$
we have
\begin{eqnarray}
P_1( r_1, r_2)^5 & = & 1 \label{r1}\\
P_2( r_1, r_2)^3 & = & s \label{r2},
\end{eqnarray}
where
\begin{eqnarray*}
y_i(t) & = & r_i\left(\frac{(t^2-t+1)^3}{t^2(t-1)^2}\right)\\
          & = & r_i(s),\quad i=1,2.
\end{eqnarray*}
If we now take the derivative of the equations (\ref{r1}) and (\ref{r2}) with respect to $s$ we will obtain a dynamical system
\begin{eqnarray*}
\frac{d}{ds} X_1 & = & \frac{1}{D(X_1,X_2)}X_1(65536X_1^{10}-16896X_1^5X_2^5-11X_2^{10})\\
\frac{d}{ds} X_2 & = & \frac{1}{D(X_1,X_2)}X_2(720896X_1^{10}-16896X_1^5X_2^5-X_2^{10}),
\end{eqnarray*}
where
{\small
\begin{eqnarray*}
D(X_1,X_2) & = & \frac{44814958964215245}{35184372088832}X_1^{70}
-\frac{1478893645819103085}{4503599627370496}X_1^{65}X_2^5\\
 & &\ 
-\frac{1507303767249340213172175}{2305843009213693952}X_1^{60}X_2^{10}\\
 & &\ \ 
-\frac{200131326479517435045}{35184372088832}X_1^{55}X_2^{15}\\
 & & \ \ \
-\frac{41163273776082534986385}{72057594037927936}X_1^{50}X_2^{20}\\
 & & \ \ \ \
-\frac{76393418368808015853255}{9223372036854775808}X_1^{45}X_2^{25}\\
 & & \ \ \ \ \
-\frac{119880246375756138266115}{4722366482869645213696}X_1^{40}X_2^{30}\\
 & & \ \ \ \ \ \  
-\frac{3712236328125}{18014398509481984}X_1^{35}X_2^{35}\\
 & & \ \ \ \ \ \ \ 
-\frac{57031746165946245}{147573952589676412928}X_1^{30}X_2^{40}\\
 & & \ \ \ \ \ \ \ \
+\frac{36602149417814565}{18889465931478580854784}X_1^{25}X_2^{45}\\
 & & \ \ \ \ \ \ \ \ \
-\frac{20359295143555005}{9671406556917033397649408}X_1^{20}X_2^{50}\\
 & & \ \ \ \ \ \ \ \ \ \
+\frac{59979105}{147573952589676412928}X_1^{15}X_2^{55}\\
 & & \ \ \ \ \ \ \ \ \ \ \
-\frac{2925975}{302231454903657293676544}X_1^{10}X_2^{60}\\
 & & \ \ \ \ \ \ \ \ \ \ \ \
+\frac{495}{38685626227668133590597632}X_1^5X_2^{65}\\
 & & \ \ \ \ \ \ \ \ \ \ \ \ \
+\frac{15}{19807040628566084398385987584}X_2^{70},
\end{eqnarray*}
}with solution $\mathbf{X}(s)=(r_1,r_2)$. The two functions $r_1,r_2$ are solutions to the linear ordinary differential equation
$$\mathcal{L}(r)=\left(\frac{d}{ds}\right)^2r+\frac{2}{3}\frac{7s-27}{s(4s-27)}\frac{d}{ds}r-\frac{11}{900}\frac{1}{s(4s-27)}r=0$$
and
$$\Pi_{G,\mathbb{P}}(r_1,r_2)=(1:s).$$
\end{exa}

\begin{defn}
Let $L(x)=0$ be a linear ordinary differential equation, where
\begin{equation*}
L(x)=\left(\frac{d}{dt}\right)^n x+a_{n-1}\left(\frac{d}{dt}\right)^{n-1}x+\ldots+a_1\frac{d}{dt}x+a_0x,
\end{equation*}
$x=x(t)$ and $a_{n-1},\ldots,a_1,a_0\in\mathbb{C}(t)$ and let $\mathcal{L}(r)=0$ be a linear ordinary differential equation, where
\begin{equation}\label{LODEs}
\mathcal{L}(r)=\left(\frac{d}{ds}\right)^n r+b_{n-1}\left(\frac{d}{ds}\right)^{n-1}r+\ldots+b_1\frac{d}{ds}r+b_0r
\end{equation}
and $b_{n-1},\ldots,b_1,b_0\in\mathbb{C}(s)$. We say that $L(x)=0$ is a pullback of $\mathcal{L}(r)$ by a rational map if there exists $p\in\mathbb{C}(t)$, such that every solution to $L(x)=0$ is of the form $x(t)=r\circ p(t)$ where $r$ is a solution to $\mathcal{L}(r)=0$. 
\end{defn}

We illustrate the concept of pullback with an example.

\begin{exa}\label{PepinIII}
We continue with equation (\ref{Pepin2}) from Example \ref{PepinII}
$$\left(\frac{d}{dt}\right)^2y+\frac{2}{3}\frac{2t-1}{t(t-1)}\frac{d}{dt}y-\frac{11}{900}\frac{t^2-t+1}{t^2(t-1)^2}y=0,$$
that is projectively equivalent to P\'epin's equation. This equation is a pullback of
\begin{equation}\label{Pepin3}
\left(\frac{d}{ds}\right)^2r+\frac{2}{3}\frac{7s-27}{s(4s-27)}\frac{d}{ds}r-\frac{11}{900}\frac{1}{s(4s-27)}r=0
\end{equation}
by the rational map $s=p(t)=(t^2-t+1)^3/(t^2-t)^2$. Equation (\ref{Pepin3}), under the M\"obius transformation $$z=\frac{4}{27}s$$ becomes the hypergeometric equation defining ${}_2 F_1(-1/60,11/60;2/3\ |\ z)$
$$\left(\frac{d}{dz}\right)^2\sigma+\frac{1}{6}\frac{7z-4}{z(z-1)}\frac{d}{dz}\sigma-\frac{11}{3600}\frac{1}{z(z-1)}\sigma=0.$$
In this way, from the pullback and the projective equivalence we can obtain closed form solutions to P\'epin's equation in terms of well-known functions. For example one solution is
\begin{eqnarray*}
x_0(t) & = & t^{1/3}(t-1)^{1/3}y_0(t)\\
           & = & t^{1/3}(t-1)^{1/3}r_0\left(\frac{(t^2-t+1)^3}{t^2(t-1)^2}\right)\\
           & = & t^{1/3}(t-1)^{1/3}\sigma_0\left(\frac{4}{27}\frac{(t^2-t+1)^3}{t^2(t-1)^2}\right)\\
           & = & (t^2-t)^{1/3}{}_2 F_1\left(\frac{-1}{60},\frac{11}{60};\frac{2}{3}\ |\ \frac{4}{27}\frac{(t^2-t+1)^3}{(t^2-t)^2}\right).
\end{eqnarray*}
Another linearly independent solution can be similarly obtained from another linearly independent solution to the equation for $\sigma(z)$, for example $$\sigma(z)=z^{1/3}{}_2 F_1(19/60,31/60;4/3\ |\ z).$$
\end{exa}

\subsection{Quotient Fano curve, minimal Schwarz maps and Standard Equations}

From the Galois correspondence it follows that the composition of the Schwarz map with the projection into the projective orbits
$$\Pi_{G,\mathbb{P}}\circ\left[\mathbf{X}\right](t)=\Pi_{G,\mathbb{P}}(x_1:\ldots:x_n)$$
parametrizes  an algebraic curve in $\mathbb{P}(\mathbb{C}^{n}/G)$. Therefore we can make an analytic extension of this composition to the whole Riemann sphere to obtain a single-valued algebraic map
\begin{eqnarray*}
\Pi_{G,\mathbb{P}}\circ\left[\mathbf{X}\right]: \mathbb{P}(\mathbb{C})& \longrightarrow & \mathbb{P}(\mathbb{C}^{n}/G)\\
t & \longmapsto & \Pi_{G,\mathbb{P}}\circ\left[\mathbf{X}\right](t),
\end{eqnarray*}
where we use the same notation for the original composition and the extension.  We will call the closure of the image of the map $\Pi_{G,\mathbb{P}}\circ\left[\mathbf{X}\right]$ the quotient Fano curve. If this map is birational we say that the Schwarz map  $\left[\mathbf{X}\right](t)=(x_1:\ldots:x_n)$ is minimal \cite{KATO2006} and in such case we say that the equation $L(x)=0$ is standard \cite{BERKENBOSCH2006}.

We illustrate the concept of standard equation with two examples.

\begin{exa}
Continuing with equation (\ref{Pepin3}) from Example \ref{PepinIII}, $\mathcal{L}(r)=0$ is an standard equation but P\'epin's equation is not.
\end{exa}

\begin{exa}\label{Klein}
In this example we will consider Klein's quartic. As a Fano curve it is parametrized by the Schwarz map associated to Hurwitz' equation
\begin{eqnarray*}
0 & = & \left(\frac{d}{dt}\right)^3x+\frac{7t-4}{t(t-1)}\left(\frac{d}{dt}\right)^2x+\frac{1}{252}\frac{2592t^2-2963t+560}{t^2(t-1)^2}\left(\frac{d}{dt}\right)x \\
 & & \ +\frac{1}{24696}\frac{57024t-40805}{t^2(t-1)^2}x.
\end{eqnarray*}
This equation has differential Galois group $G_{168}$ \cite{SINGER1995}. Computing its invariants we obtain
\begin{eqnarray*}
P_4(x_1,x_2,x_3) & = & 0\\
P_6(x_1,x_2,x_3) & = & \frac{1}{(t-1)^3t^4}\\
P_{14}(x_1,x_2,x_3) & = & \frac{1}{(t-1)^7t^9}.
\end{eqnarray*}
As homogeneous coordinates of the space of projective orbits we take three invariants of degree $42$, $P_4^9P_6$, $P_6^7$, $P_{14}^3$, so that
\begin{eqnarray*}
\Pi_{G,\mathbb{P}}(x_1,x_2,x_3) & = & (0:\ \frac{1}{(t-1)^{21}t^{28}}:\ \frac{1}{(t-1)^{21}t^{27}})\\
 & = & (0:\ 1:\ t).
\end{eqnarray*}
Hence Hurwitz' equation is standard, and the associated Schwarz map is minimal.
\end{exa}

The importance of standard equation is stated in the following theorem \cite[Corollary 12]{SANABRIA2014}.

\begin{thm}
Let $L(x)=0$ be an algebraic irreducible linear ordinary differential equation. Then $L(x)=0$ is rationally projectively equivalent to the pullback by a rational map of an standard equation.
\end{thm}

\subsection{Associated first order differential equation, dynamical systems and standard equations}

Let us consider the first order differential equation (\ref{difeq}) built in Section \ref{1stDE} from the system in (\ref{difgen}). If $L(x)=0$ is standard, then this differential equation is a dynamical systems.

\begin{defn}\label{defasso}
Let $L(x)=0$ be an algebraic irreducible linear ordinary differential equation. Let $x_1,\ldots,x_n$ be $n$ linearly independent solutions to $L(x)=0$ and $G$ the representation of the differential Galois group of $L(x)=0$ associated to this choice of solutions. Let $\Lambda\in\mathbb{Z}_{>0}$ be such that the homogeneous elements of $\mathbb{C}[X_1,\ldots,X_n]^G$ of degree $\Lambda$, $\mathbb{C}[X_1,\ldots,X_n]^G_\Lambda$, form a homogeneous coordinate system for $\mathbb{P}(\mathbb{C}^n/G)$, i.e. $$\mathbb{P}(\mathbb{C}[X_1,\ldots,X_n]^G_\Lambda)\simeq\mathbb{P}(\mathbb{C}^n/G).$$ Let $P_1,\ldots,P_n\in\mathbb{C}[X_1,\ldots,X_n]^G_\Lambda$ be such that 
\begin{equation}\label{assoM}
M(\mathbf{X})=\left[\begin{array}{ccc}
\partial P_1/\partial X_1 & \ldots & \partial P_1/\partial X_n\\
\vdots & \ddots & \vdots\\
\partial P_n/\partial X_1 & \ldots & \partial P_n/\partial X_n
\end{array}\right](\mathbf{X})
\end{equation}
is invertible, and let $f_1,\ldots,f_n\in\mathbb{C}(t)$ be such that $$f_i=P_i(x_1,\ldots,x_n),\quad i=1,\ldots,n.$$ The first order differential equation associated  to $L(x)=0$, given the solutions $x_1,\ldots, x_n$, the degree $\Lambda$, and the invariants $P_1,\ldots,P_n$ is
\begin{equation}\label{assodifeq}
\frac{d}{dt}\mathbf{X}=F(t,\mathbf{X})
\end{equation}
where
\begin{equation}\label{assoFdifeq}
F(t,\mathbf{X})=\big[M(\mathbf{X})\big]^{-1}\left[\begin{array}{c} \frac{d}{dt}f_1\\ \vdots \\ \frac{d}{dt}f_n\end{array}\right].
\end{equation}
\end{defn}

\begin{exa}
In Example \ref{PepinII}, we constructed the first order differential equation associated to $\mathcal{L}(r)=0$ given the solutions $r_1,r_2$, the degree $60$, and the invariants $P_1^5,P_2^3$.
\end{exa}

In the following theorem we characterize standard equations, up to projective equivalence, in terms of their associated first order differential equation.

\begin{thm}\label{thmI}
Let $L(x)=0$ be an algebraic irreducible linear ordinary differential equation. Then $L(x)=0$ is standard if and only if it is rationally projectively equivalent to a linear ordinary differential equation with an associated first order differential equation that is a dynamical system.
\end{thm}

\begin{proof} Let us assume that $L(x)=0$ is standard. We fix $n$ linearly independent solutions $x_1,\ldots,x_n$. By the definition of standard equation, the map $$t\mapsto \Pi_{G,\mathbb{P}}\circ\left[\mathbf{X}\right](t)$$ defines a birational morphism between the Riemann sphere parameterized by $t$ and the quotient Fano curve. Let us denote by $C$ the quotient Fano curve. Then the map
\begin{eqnarray*}
\Pi_{G,\mathbb{P}}\circ\left[\mathbf{X}\right]: \mathbb{P}(\mathbb{C})& \longrightarrow & C\subseteq\mathbb{P}(\mathbb{C}^{n}/G)\\
t & \longmapsto & \Pi_{G,\mathbb{P}}\circ\left[\mathbf{X}\right](t)
\end{eqnarray*}
is birational.

The homogeneous ideal in $\mathbb{C}[X_1,\ldots,X_n]^G$ defining $C$ is generated by the homogeneous elements in $I\cap \mathbb{C}[X_1,\ldots,X_n]^G$
\begin{eqnarray*}
I(C) & = & \left(I\cap \mathbb{C}[X_1,\ldots,X_n]^G\right)^h\\
      & = & \langle P\in\mathbb{C}[X_1,\ldots,X_n]^G |\ P\textrm{ is homogeneous and } P(x_1,\ldots,x_n)=0\rangle.
\end{eqnarray*}
The homogeneous coordinate ring of $C$ is $\mathbb{C}[X_1,\ldots,X_n]^G/I(C)$. Since $C$ is a curve, there exist $n-2$ polynomials $Q_1,\ldots,Q_{n-2}\in I(C)$ that are algebraically independent over $\mathbb{C}$.

Furthermore, since $\Pi_{G,\mathbb{P}}\circ\left[\mathbf{X}\right]$ is a birational morphism, the associated dual map is an isomorphism from the field of meromorphic functions $\mathbb{C}(C)$ to $\mathbb{C}(t)$. In particular, there exist two homogeneous polynomials of the same degree, $P,Q\in\mathbb{C}[X_1,\ldots,X_n]^G$, such that $$\frac{P(x_1,\ldots,x_n)}{Q(x_1,\ldots,x_n)}=t.$$

Let $$f=Q(x_1,\ldots,x_n)\in\mathbb{C}(t).$$ Note that $f\ne 0$, and that $Q_1,\ldots,Q_{n-2},Q,P$ can be taken so that they have the same degree $\Lambda$.

Now we consider $L_1(y)$, the linear ordinary differential equation projectively equivalent to $L(x)=0$ with solutions $$y_i(t)=f^{-1/\deg(Q)}x_i(t),\quad i=1,\ldots,n.$$
We have $$Q(y_1,\ldots,y_n)  =  1 \textrm{ and } \frac{P(y_1,\ldots,y_n)}{Q(y_1,\ldots,y_n)}  =  t,$$
hence
\begin{eqnarray*}
Q_i(y_1,\ldots,y_n) & = & 0, \textrm{ for } i=1,\ldots,n-2,\\
Q(y_1,\ldots,y_n)  & = &  1,\\
P(y_1,\ldots,y_n) & = & t.
\end{eqnarray*}
Therefore, the first order differential equation associated to $L_1(y)=0$ given the solutions $y_1,\ldots,y_n$, the degree $\Lambda$, and the invariants $P_i=Q_i$ for $i=1,\ldots, n-2$, $P_{n-1}=Q$ and $P_n=P$ is a dynamical system. This finishes the first part of the proof.

Now we assume that $L(x)=0$ is projectively equivalent to $L_1(y)=0$ and that the first order differential equation associated to $L_1(y)=0$ given the solutions $y_1,\ldots,y_n$, the degree $\Lambda$ and the invariants $P_1,\ldots,P_n$, is a dynamical system
\begin{equation*}\label{assodifeqnot}
\frac{d}{dt}\mathbf{X}=F(\mathbf{X})
\end{equation*}
where $$f_i=P_i(y_1,\ldots,y_n)\in\mathbb{C}(t),\quad i=1,\ldots,n,$$ and
\begin{equation*}\label{assoFdifeqnot}
F(\mathbf{X})=\big[M(\mathbf{X})\big]^{-1}\left[\begin{array}{c} \frac{d}{dt}f_1\\ \vdots \\ \frac{d}{dt}f_n\end{array}\right].
\end{equation*}
Since the right hand side of the last equation doesn't depend on $t$, we have $$\frac{d}{dt}f_i\in\mathbb{C},\ i=1,\ldots,n$$ hence $f_i=\alpha_it+\beta_i,\ \alpha_i,\beta_i\in\mathbb{C},\ i=1,\ldots,n$.

Let $G_1$ be the representation of the differential Galois group of $L_1(y)=0$ given by $y_1,\ldots ,y_n$, and $$d=\dim_{\mathbb{C}}\left(\mathbb{C}[X_1,\ldots,X_n]^{G_1}_\Lambda\right).$$ We can complete the elements  $P_1,\ldots,P_n$ to form a basis $(P_1,\ldots,P_n,\ldots,P_{d})$ of the $\mathbb{C}$-vector space $\mathbb{C}[X_1,\ldots,X_n]^{G_1}_\Lambda$. In particular $(P_1:\ldots:P_n:\ldots:P_{d})$ is a homogeneous coordinate system for $\mathbb{P}(\mathbb{C}^n/G_1)$.

Let us denote $\left[\mathbf{Y}\right](t)=(y_1:\ldots:y_n)$. If we denote 
$f_i=P_i(y_1,\ldots,y_n)$, for $i=n+1,\ldots,d,$ then the image of the map $\Pi_{G_1,\mathbb{P}}\circ\left[\mathbf{Y}\right]$ in the homogeneous coordinate system $(P_1:\ldots:P_n:P_{n+1}:\ldots:P_{d})$ is $$\Pi_{G_1,\mathbb{P}}\circ\left[\mathbf{Y}\right](t)=(\alpha_1t+\beta_1:\ldots:\alpha_nt+\beta_n:f_{n+1}:\ldots:f_d).$$ Hence $\Pi_{G_1,\mathbb{P}}\circ\left[\mathbf{Y}\right]$ is a birational morphism and $L_1(y)=0$ is standard. The proof is completed by noting that the morphism to the quotient Fano curve is the same for two rationally projective linear ordinary differential equation, therefore $L(x)=0$ is also standard.
\end{proof}

We illustrate the theorem with the following example.

\begin{exa}\label{KleinII}
Hurwitz's equation in Example \ref{Klein} is projectively equivalent to
\begin{eqnarray*}
0 & = & \left(\frac{d}{dt}\right)^3y+\frac{1}{2}\frac{7t-4}{t(t-1)}\left(\frac{d}{dt}\right)^2y+\frac{1}{252}\frac{387t-56}{t^2(t-1)}\left(\frac{d}{dt}\right)y \\
 & & \ -\frac{85}{74088}\frac{1}{t^2(t-1)}y.
\end{eqnarray*}
The solutions to this equations are of the form $$y_0(t)=(t-1)^{1/2}t^{2/3}x_0(t)$$ where $x_0$ is a solution to Hurwitz's equation. In particular we have the system
 \begin{eqnarray}
P_4^9P_6(y_1,y_2,y_3) & = & 0 \label{KleinII1}\\
P_6^7(y_1,y_2,y_3) & = & 1 \label{KleinII2}\\
P_{14}^3(y_1,y_2,y_3) & = & t \label{KleinII3}
\end{eqnarray}
and therefore the associated first order differential equations given the solutions $y_i(t)=(t-1)^{1/2}t^{2/3}x_i(t)$, for $i=1,2,3$, the degree $42$ and the invariants $P_4^9P_6$, $P_6^7$, $P_{14}^3$ is a dynamical system.
\end{exa}

\section{Differential dynamical systems invariant under finite linear algebraic group actions}

In the previous section we showed how, starting with an algebraic irreducible linear ordinary differential equation, we can construct a first order ordinary differential equation having as solutions a full system of solutions to the original equation. Moreover, if the resulting ordinary differential equation is a dynamical system, then the original equation is a standard equation.

In this section we will characterize the differential dynamical systems associated to standard equations. Furthermore, we will show that every such dynamical system will give rise to a family of standard equations, corresponding to the different orbits.

For the rest of this section we will assume that $G\subset GL_n(\mathbb{C})$ is a finite group.

\subsection{$G$-homogeneous dynamical systems}

\begin{defn}\label{defGhomo}
Let $G\subset GL_n(\mathbb{C})$ be a finite group and
\begin{equation}\label{dynsys}
\frac{d}{dt}\mathbf{X}=F(\mathbf{X})
\end{equation}
a dynamical system such that every coordinate function of  $F=(F_1,\ldots,F_n)$ is rational, i.e. $F_i\in\mathbb{C}(X_1,\ldots,X_n)$, for $i=1,\ldots,n$. We say that this dynamical system is $G$-homogeneous of degree $\Lambda\in\mathbb{Z}_{>0}$ if there exist $n$ homogeneous polynomials $P_1,\ldots,P_n\in\mathbb{C}[X_1,\ldots,X_n]^G$, algebraically independent over $\mathbb{C}$, and of degree $\Lambda$ such that if $\mathbf{X}(t)$ is a solution of the system, then
$$ \frac{d}{dt}P_i(\mathbf{X})(t)=\alpha_i\in\mathbb{C}, \ i=1,\ldots,n.$$
\end{defn}

Note that in this case
$$ F(\mathbf{X})=\left\{\left[\begin{array}{ccc}
\partial P_1/\partial X_1 & \ldots & \partial P_1/\partial X_n\\
\vdots & \ddots & \vdots\\
\partial P_n/\partial X_1 & \ldots & \partial P_n/\partial X_n
\end{array}\right](\mathbf{X})\right\}^{-1}
\left[\begin{array}{c} \alpha_1\\ \vdots\\ \alpha_n\end{array}\right].$$

We construct a $G$-homogeneous dynamical system in the next example.

\begin{exa}\label{KleinIII}
The differential Galois group of Hurwitz's equation in Example \ref{Klein} is Klein's simple group of order $168$, which is isomorphic to the subgroup $G$ of $SL_3(\mathbb{C})$ generated by the matrices \cite{SINGER1993}
\[
\left[\begin{array}{ccc}
\beta & 0 & 0\\
0 & \beta^2 & 0\\
0 & 0 & \beta^4
\end{array}\right],\quad
\left[\begin{array}{ccc}
0 & 1 & 0\\
0 & 0 & 1\\
1 & 0  & 0
\end{array}\right],\textrm{ and }
\left[\begin{array}{ccc}
a  & b & c\\
b & c & a\\
c & a  & b
\end{array}\right],
\]
where $\beta$ is a primitive $7$th root of unity, i.e. $\beta^6+\beta^5+\beta^4+\beta^3+\beta^2+\beta^1+1=0$, $a=\beta^4-\beta^3$, $b=\beta^2-\beta^5$ and $c=\beta-\beta^6$. Three algebraically independent invariants in $\mathbb{C}[X_1,X_2,X_3]$ are \cite{KATO2004}
\begin{eqnarray*}
P_4 & = & X_1^3X_2+X_2^3X_3+X_3^3X_1,\\
P_6 & = & \frac{1}{54}\det\left[\begin{array}{ccc}
\partial^2 F_4/\partial X_1\partial X_1 & \partial^2 F_4/\partial X_1\partial X_2 & \partial^2 F_4/\partial X_1\partial X_3 \\
\partial^2 F_4/\partial X_2\partial X_1 & \partial^2 F_4/\partial X_2\partial X_2 & \partial^2 F_4/\partial X_2\partial X_3 \\
\partial^2 F_4/\partial X_3\partial X_1 & \partial^2 F_4/\partial X_3\partial X_2 & \partial^2 F_4/\partial X_3\partial X_3 \\
\end{array}\right],\\
P_{14} & =  & \frac{1}{9}\det\left[\begin{array}{cccc}
\partial^2 F_4/\partial X_1\partial X_1 & \partial^2 F_4/\partial X_1\partial X_2 & \partial^2 F_4/\partial X_1\partial X_3 & \partial F_6/\partial X_1\\
\partial^2 F_4/\partial X_2\partial X_1 & \partial^2 F_4/\partial X_2\partial X_2 & \partial^2 F_4/\partial X_2\partial X_3 & \partial F_6/\partial X_2\\
\partial^2 F_4/\partial X_3\partial X_1 & \partial^2 F_4/\partial X_3\partial X_2 & \partial^2 F_4/\partial X_3\partial X_3 & \partial F_6/\partial X_3\\
\partial F_6/\partial X_1 & \partial F_6/\partial X_2 & \partial F_6/\partial X_3 & 0
\end{array}\right].
\end{eqnarray*}
Based on equations (\ref{KleinII1}), (\ref{KleinII2}), (\ref{KleinII3}) from Example \ref{KleinII},  we set
 \begin{eqnarray*}
\frac{d}{dt}P_4^9P_6(X_1,X_2,X_3) & = & 0\\
\frac{d}{dt}P_6^7(X_1,X_2,X_3) & = & 0\\
\frac{d}{dt}P_{14}^3(X_1,X_2,X_3) & = & 1
\end{eqnarray*}
to obtain the $G$-homogeneous dynamical system of degree $42$
\begin{eqnarray*}
\frac{d}{dt} X_1 & = & \frac{1}{\Delta}\left(\frac{\partial P_4^9P_6}{\partial X_2}\frac{\partial P_6^7}{\partial X_3}-\frac{\partial P_4^9P_6}{\partial X_3}\frac{\partial P_6^7}{\partial X_2}\right)(X_1,X_2,X_3)\\
\frac{d}{dt} X_2 & = & \frac{1}{\Delta}\left(\frac{\partial P_4^9P_6}{\partial X_1}\frac{\partial P_6^7}{\partial X_3}-\frac{\partial P_4^9P_6}{\partial X_3}\frac{\partial P_6^7}{\partial X_1}\right)(X_1,X_2,X_3)\\
\frac{d}{dt} X_1 & = & \frac{1}{\Delta}\left(\frac{\partial P_4^9P_6}{\partial X_1}\frac{\partial P_6^7}{\partial X_2}-\frac{\partial P_4^9P_6}{\partial X_2}\frac{\partial P_6^7}{\partial X_1}\right)(X_1,X_2,X_3)
\end{eqnarray*}
where
\begin{eqnarray*}
\Delta & = & \det\left[\begin{array}{ccc}
\partial P_4^9P_6/\partial X_1 & \partial P_4^9P_6/\partial X_2 & \partial P_4^9P_6/\partial X_3 \\
\partial P_6^7/\partial X_1 & \partial P_6^7/\partial X_2 & \partial P_6^7/\partial X_3 \\
\partial P_{14}^3/\partial X_1 & \partial P_{14}^3/\partial X_2 & \partial P_{14}^3/\partial X_3
\end{array}\right].
\end{eqnarray*}
\end{exa}

\subsection{Linear ordinary differential equations for orbits of $G$-homogeneous dynamical system}

In this section we will show how the coordinates $x_1,\ldots,x_n$ of each solution $\mathbf{X}(t)$ to a $G$-homogeneous dynamical system form a full system of solutions to an algebraic linear ordinary differential equation.

Let $P_1,\ldots,P_n\in\mathbb{C}[X_1,\ldots,X_n]^G$ be algebraically independent over $\mathbb{C}$ and $\alpha_1,\ldots,\alpha_n\in\mathbb{C}$. Then, for any $\beta_1,\ldots,\beta_n\in\mathbb{C}$, the $\mathbb{C}$-morphism
\begin{eqnarray*}
\mathbb{C}[P_1,\ldots,P_n] &\longrightarrow & \mathbb{C}(t)\\
P_i & \longmapsto & \alpha_i t+\beta_i,\ i=1,\ldots,n
\end{eqnarray*}
can be extended to a morphism
\begin{eqnarray}\label{extension}
\Phi: \mathbb{C}[X_1,\ldots,X_n]^G &\longrightarrow & \overline{\mathbb{C}(t)}.
\end{eqnarray}

Now let $\frac{d}{dt}\mathbf{X}=F(\mathbf{X})$ be a $G$-homogeneous dynamical system of degree $\Lambda$, and let $P_1,\ldots,P_n$ and $\alpha_1,\ldots,\alpha_n$ be as in Definition \ref{defGhomo}. Let us fix a solution $\mathbf{X}(t)=(x_1,\ldots,x_n)$ to the dynamical system. Then there exist $\beta_1,\ldots,\beta_n\in\mathbb{C}$ such that $$P_i(x_1,\ldots,x_n)=\alpha_i t+\beta_i, \quad i=1,\ldots,n,$$
hence
$$\frac{d}{dt}\left[\begin{array}{c} x_1\\ \vdots\\ x_n\end{array}\right]=\left\{\left[\begin{array}{ccc}
\partial P_1/\partial X_1 & \ldots & \partial P_1/\partial X_n\\
\vdots & \ddots & \vdots\\
\partial P_n/\partial X_1 & \ldots & \partial P_n/\partial X_n
\end{array}\right](x_1,\ldots,x_n)\right\}^{-1}
\left[\begin{array}{c} \alpha_1\\ \vdots\\ \alpha_n\end{array}\right].$$
Differentiating with respect to $t$ we obtain
{\footnotesize
\begin{eqnarray*}
\left(\frac{d}{dt}\right)^2\left[\begin{array}{c} x_1\\ \vdots\\ x_n\end{array}\right] & = & D\left(\left\{\left[\begin{array}{ccc}
\partial P_1/\partial X_1 & \ldots & \partial P_1/\partial X_n\\
\vdots & \ddots & \vdots\\
\partial P_n/\partial X_1 & \ldots & \partial P_n/\partial X_n
\end{array}\right]\right\}^{-1}
\left[\begin{array}{c} \alpha_1\\ \vdots\\ \alpha_n\end{array}\right]\right)(x_1,\ldots,x_n) \cdot\\
&  &\quad \frac{d}{dt}\left[\begin{array}{c} x_1\\ \vdots\\ x_n\end{array}\right]\\
& = & D\left(\left\{\left[\begin{array}{ccc}
\partial P_1/\partial X_1 & \ldots & \partial P_1/\partial X_n\\
\vdots & \ddots & \vdots\\
\partial P_n/\partial X_1 & \ldots & \partial P_n/\partial X_n
\end{array}\right]\right\}^{-1}
\left[\begin{array}{c} \alpha_1\\ \vdots\\ \alpha_n\end{array}\right]\right)(x_1,\ldots,x_n) \cdot\\
& &\quad  \left\{\left[\begin{array}{ccc}
\partial P_1/\partial X_1 & \ldots & \partial P_1/\partial X_n\\
\vdots & \ddots & \vdots\\
\partial P_n/\partial X_1 & \ldots & \partial P_n/\partial X_n
\end{array}\right](x_1,\ldots,x_n)\right\}^{-1}
\left[\begin{array}{c} \alpha_1\\ \vdots\\ \alpha_n\end{array}\right],
\end{eqnarray*}}
where $D$ is the total derivative with respect to $(X_1,\ldots,X_n)$.

Repeating this process we obtain rational expressions for $\left(\frac{d}{dt}\right)^{i}x_j$, for $i=0,1,\ldots,n$ and $j=1,\ldots,n$ in terms of $x_1,\ldots,x_n$. Hence for every value of $t$ where the solution $\mathbf{X}(t)=(x_1,\ldots,x_n)$ is defined we have $n+1$ vector in $\mathbb{C}^n$
\begin{equation*}\label{diffvecs}
\left(\frac{d}{dt}\right)^n\left[\begin{array}{c} x_1\\ \vdots\\ x_n\end{array}\right],\quad \ldots \ ,\quad \frac{d}{dt}\left[\begin{array}{c} x_1\\ \vdots\\ x_n\end{array}\right], \quad
\left[\begin{array}{c} x_1\\ \vdots\\ x_n\end{array}\right].
\end{equation*}
To find the linear dependence relation between them, we rewrite their coordinates in terms of $(x_1,\ldots,x_n)$ by using the expressions we obtained through iterated differentiation with respect to $t$. As the coefficients in this linear system are in the field $\mathbb{C}(x_1,\ldots,x_n)$ so is their solution. Hence there exist $Q_n,\ldots,Q_1,Q_0\in\mathbb{C}(x_1,\ldots,x_n)$ such that
\begin{equation}\label{lindepdiffvecs}
0=Q_n\left\{\left(\frac{d}{dt}\right)^n\left[\begin{array}{c} x_1\\ \vdots\\ x_n\end{array}\right]\right\}+\ldots+
Q_1\left\{\frac{d}{dt}\left[\begin{array}{c} x_1\\ \vdots\\ x_n\end{array}\right]\right\}+Q_0
\left\{\left[\begin{array}{c} x_1\\ \vdots\\ x_n\end{array}\right]\right\}.
\end{equation}

Now, since each $P_1,\ldots,P_n$ is $G$-invariant, and the action of $G$ as well as differentiation are $\mathbb{C}$-linear, the $G$-action permutes the solutions of the $G$-homogeneous dynamical system. Therefore, the linear dependence relation in (\ref{lindepdiffvecs}) is $G$-invariant and $Q_n,\ldots,Q_1,Q_0\in\mathbb{C}(x_1,\ldots,x_n)^G$.

The evaluations $P_i(x_1,\ldots,x_n)=\alpha_i t+\beta_i$, for  $i=1,\ldots,n$ can be extended via the $\mathbb{C}$-morphism $\Phi$ in (\ref{extension}) to the field $\mathbb{C}(x_1,\ldots,x_n)^G$. Therefore, we obtain $c_n,\ldots,c_1,c_0\in\overline{\mathbb{C}(t)}$ such that
$c_i=Q_i(x_1,\ldots,x_n)$, for $i=1,\ldots,n$ and
$$c_n\left\{\left(\frac{d}{dt}\right)^n\left[\begin{array}{c} x_1\\ \vdots\\ x_n\end{array}\right]\right\}+\ldots+
c_1\left\{\frac{d}{dt}\left[\begin{array}{c} x_1\\ \vdots\\ x_n\end{array}\right]\right\}+c_0
\left\{\left[\begin{array}{c} x_1\\ \vdots\\ x_n\end{array}\right]\right\}=0.$$
Hence the coordinates of $\mathbf{X}(t)=(x_1,\ldots,x_n)$ are solutions to
$$c_n\left(\frac{d}{dt}\right)^n x+c_{n-1}\left(\frac{d}{dt}\right)^{n-1}x+\ldots+c_1\frac{d}{dt}x+c_0x=0.$$

We illustrate the evaluation map in (\ref{extension}) in the following example.

\begin{exa}
We continue with Example \ref{KleinIII}. In it we obtained a $G$-homogeneous dynamical system of degree $42$. Let us consider the solution $\mathbf{Y}(t)=(y_1,y_2,y_3)$ with initial conditions
$$P_4^9P_6(\mathbf{Y})(0)=0,\quad P_6^7(\mathbf{Y})(0)=\frac{1}{1728},\quad P_{14}^3(\mathbf{Y})(0)=0.$$
Hence
$$P_4^9P_6(\mathbf{Y})(t)=0,\quad P_6^7(\mathbf{Y})(t)=\frac{1}{1728},\quad P_{14}^3(\mathbf{Y})(t)=t$$
and
$$P_4(\mathbf{Y})(t)=0,\quad P_6(\mathbf{Y})(t)=1728^{-1/7},\quad P_{14}(\mathbf{Y})(t)=t^{1/3}.$$
Therefore the linear ordinary differential equation satisfied by $y_1$, $y_2$, and $y_3$ is the equation in Example \ref{KleinII}.
\end{exa}

We now prove that the coordinates of the solutions to $G$-homogeneous dynamical systems are algebraic by showing that the differential Galois group of the equation is finite.

\begin{thm}\label{thmII}
Let $G\subset GL_n(\mathbb{C})$ be a finite group and $\frac{d}{dt}\mathbf{X}=F(\mathbf{X})$ a $G$-homogeneous dynamical system. Let $P_1,\ldots,P_n\in\mathbb{C}[X_1,\ldots,X_n]^G$  and let $\alpha_1,\ldots,\alpha_n\in\mathbb{C}$ be as in Definition \ref{defGhomo}. Let $\mathbf{X}(t)=(x_1,\ldots, x_n)$ be a solution to this $G$-homogeneous dynamical system and let $\beta_1,\ldots,\beta_n\in\mathbb{C}$ be such that $$P_i(x_1,\ldots,x_n)=\alpha_i t+\beta_i, \quad i=1,\ldots,n.$$ Let $K_0\subseteq\overline{\mathbb{C}(t)}$ be the field generated by the image of the extension $\Phi$ in (\ref{extension}). Let
 $$c_n\left(\frac{d}{dt}\right)^n x+c_{n-1}\left(\frac{d}{dt}\right)^{n-1}x+\ldots+c_1\frac{d}{dt}x+c_0x=0$$
be the linear ordinary differential equation satisfied by $x_1,\ldots,x_n$, and let $G_1\in GL_n(\mathbb{C})$ be the representation of the differential Galois group of this equation over $K_0$ given by the solutions $x_1,\ldots,x_n$. Then, if $x_1,\ldots,x_n$ are linearly independent over $\mathbb{C}$, we have $G_1=G$.  
\end{thm}

\begin{proof}
Under the evaluation $K_0$-morphism
\begin{eqnarray*}
\Psi: K_0\Big[X_{ij}|\ i,j=1,\ldots,n\Big]\left[\frac{1}{\det(X_{ij})}\right] & \longrightarrow & K\\
X_{ij} & \longmapsto & \frac{d^{i-1}}{dt^{i-1}}x_j,
\end{eqnarray*}
where
$$K=K_0\left(\left(\frac{d}{dt}\right)^{i-1} x_j \Big|\ i,j=1,\ldots,n\right),$$
every $G$-invariant polynomial is sent to an element in $K_0$. Therefore $K^G= K_0=K^{G_1}$. Using the Galois correspondence we obtain $G_1=G$.
\end{proof}

In the context of Theorem \ref{thmII}, the following example shows that the representation of the Galois group over $\mathbb{C}(t)$ given by the solutions $x_1,\ldots,x_n$ may contain $G$ as a proper subgroup.

\begin{exa}
Consider the group $A_4^{SL_2}$, isomorphic to the subgroup $G$ of $SL_2(\mathbb{C})$  generated by \cite{SINGER1993}
\[
\left[\begin{array}{cc}
i & 0 \\
0 & -i
\end{array}\right]\textrm{ and }
\left[\begin{array}{cc}
\frac{i-1}{2} & \frac{i-1}{2}\\
\frac{i+1}{2} & -\frac{i+1}{2}
\end{array}\right].
\]
Two algebraically independent elements of $\mathbb{C}[X_1,X_2]^G$ are
\begin{eqnarray*}
P_6 & = & X_1^5X_2-X_1X_2^5,\\
P_8 & = & X_1^8+14X_1^4X_2^4+X_2^8.
\end{eqnarray*}
The $G$-homogeneous dynamical system of degree $24$ defined by the relations
\begin{eqnarray*}
\frac{d}{dt}P_6^4(X_1,X_2) & = &  1\\
\frac{d}{dt}P_8^3(X_1,X_2) & = & 0
\end{eqnarray*}
will yield the linear ordinary differential equation
$$ \left(\frac{d}{dt}\right)^2 x+\frac{1}{4}\frac{5t-3}{t(t-1)}\frac{d}{dt}x-\frac{7}{576}\frac{1}{t(t-1)}x=0 $$
corresponding to the solution $\mathbf{X}(t)=(x_1,x_2)$ such that
\begin{eqnarray*}
P_6^4(x_1,x_2) & = &  t\\
P_8^3(x_1,x_2) & = & 108.
\end{eqnarray*}
We have $$K_0=\mathbb{C}(P_6,P_8)=\mathbb{C}(t^{1/4}).$$ From Theorem \ref{thmII}, the differential Galois group of the equation over $K_0$ is $G$. Now, from the tower $$\mathbb{C}(t)\subset\mathbb{C}(t^{1/4})\subset\mathbb{C}(x_1,x_2)$$ we obtain that the differential Galois group over $\mathbb{C}(t)$ is an extension of $G$ by the non-cyclic group of order $4$.
\end{exa}

\section{Example: Hesse pencil}

In this section, let $G\subset SL_3(\mathbb{C})$ be the group of order $27$ generated by the matrices \cite{ROTILLON1981}
\[
\left[\begin{array}{ccc}
1 & 0 & 0\\
0 & \omega & 0\\
0 & 0 & \omega^2
\end{array}\right]\textrm{ and }
\left[\begin{array}{ccc}
0 & 1 & 0\\
0 & 0 & 1\\
1 & 0 & 0
\end{array}\right],
\]
where $\omega$ is a primitive $3$rd root of unity, i.e. $\omega^2+\omega+1=0$. Three algebraically independent invariants in $\mathbb{C}[X_1,X_2,X_3]$ are
\begin{eqnarray*}
P_3 & = & X_1^3+X_2^3+X_3^3,\\
Q_3 & = & X_1X_2X_3,\\
P_6 & = & X_1^3X_2^3+X_2^3X_3^3+X_3^3X_1^3.
\end{eqnarray*}
Hesse pencil is the collection of $G$-invariant cubic curves
$$C(\mu:\lambda)=\left\{(X_1:X_2:X_3)\in\mathbb{P}^2(\mathbb{C})\Big|\ (\mu P_3+\lambda Q_3)(X_1:X_2:X_3)=0\right\}.$$
We have the generic case $C(1:\lambda)$ that is an elliptic curve with j-invariant 
$$j(1:\lambda)=-\lambda^3\frac{(\lambda^3-216)^3}{(\lambda^3+27)^3}$$
if $\lambda\ne -3,-3e^{2\pi i/3},-3e^{-2\pi i/3}$, and the singular cases $C(1:-3)$, $C(1:-3e^{2\pi i/3})$ and $C(1:-3e^{-2\pi i/3})$ that are the union of a line and a rational curve, and $C(0:1)$ that is the union of three lines.

We will apply the methods from this paper to obtain linear ordinary differential equations whose associated Schwarz maps $[\mathbf{X}_{\lambda}](t)$ parametrizes $C(1:\lambda)$. Let us consider the $G$-homogeneous dynamical system of degree $6$ defined by the equations
\begin{eqnarray}
P_3Q_3(x_{\lambda,1},x_{\lambda,2},x_{\lambda,3}) & = & -\lambda \label{Hesse1}\\
P_6(x_{\lambda,1},x_{\lambda,2},x_{\lambda,3}) & = & t \label{Hesse2}\\
Q_3^2(x_{\lambda,1},x_{\lambda,2},x_{\lambda,3}) & = & 1. \label{Hesse3}
\end{eqnarray}
It yields the family of algebraic linear differential equations
\begin{eqnarray*}
0 & = & \left(\frac{d}{dt}\right)^3x+\frac{3(6t^2-\lambda^2t+9\lambda)}{4t^3-\lambda^2t^2+18\lambda t-4\lambda^3+27}\left(\frac{d}{dt}\right)^2x\\
 &  &\ \ +\ \frac{8}{9}\ \frac{12t-\lambda^2}{4t^3-\lambda^2t^2+18\lambda t-4\lambda^3+27}\left(\frac{d}{dt}\right)x\\
 & &\ \ \ \ -\ \frac{8}{27}\ \frac{1}{4t^3-\lambda^2t^2+18\lambda t-4\lambda^3+27}x.
\end{eqnarray*}
The generic equation has four singularities. The discriminant of the denominator $4t^3-\lambda^2t^2+18\lambda t-4\lambda^3+27$ is $-16(\lambda^3+27)$. In the three singular cases, the discriminant vanishes, the equation has three singularities and the equation factors as the product of one of second order and one of first order.

From equations (\ref{Hesse1}), (\ref{Hesse2}), (\ref{Hesse3}) we can obtain two different $G$-homogeneous dynamical systems of degree $6$. The first one differentiating the equations with respect to $t$
\begin{eqnarray*}
\frac{d}{dt} P_3Q_3(X_1,X_2,X_3) & = & 0\\
\frac{d}{dt} P_6(X_1,X_2,X_3) & = & 1\\
\frac{d}{dt} Q_3^2(X_1,X_2,X_3) & = & 0
\end{eqnarray*}
and the other one with respect to $\lambda$
\begin{eqnarray*}
\frac{d}{d\lambda} P_3Q_3(X_1,X_2,X_3) & = &-1\\
\frac{d}{d\lambda} P_6(X_1,X_2,X_3) & = & 0\\
\frac{d}{d\lambda} Q_3^2(X_1,X_2,X_3) & = & 0.
\end{eqnarray*}
With these two $G$-homogeneous dynamical system we can study the family of Schwarz maps $[\mathbf{X}_{\lambda}](t)$ in terms of the deformation parameter $\lambda$. Note that the two vector fields defining the two $G$-homogeneous dynamical systems commute, therefore we can define $$\mathbf{X}_t(\lambda)=\mathbf{X}(\lambda,t)=\mathbf{X}_{\lambda}(t)$$ and $x_{t,i}(\lambda)=x_i(\lambda,t)=x_{\lambda,i}(t)$, for $i=1,2,3$.
The map $[\mathbf{X}_t](\lambda)$ is the Schwarz map associated to the equation
\begin{eqnarray*}
0 & = & \left(\frac{d}{d\lambda}\right)^3x+
3\frac{\lambda t^5+12\lambda^2 t^3+42\lambda^3 t+3t^4+27\lambda t^2+324\lambda^2-405t}{(t^3+9\lambda t+54)(\lambda^2t^2+4\lambda^3-4t^3-18\lambda t-27)}\left(\frac{d}{d\lambda}\right)^2x\\
 & & \ \ +\frac{2}{9}\frac{4t^5+39\lambda t^3+216\lambda^2 t+297t^2+2106\lambda}{(t^3+9\lambda t+54)(\lambda^2 t^2+4\lambda^3-4t^3-18\lambda t-27)}\left(\frac{d}{d\lambda}\right)x\\
 & & \ \ \ \ +\frac{4}{27}\frac{2t^3-9\lambda t-135}{(t^3+9\lambda t+54)(\lambda^2 t^2+4\lambda^3-4t^3-18\lambda t-27)}x.
\end{eqnarray*}
The generic equation has five singularities ($\lambda=-(t^3+54)/9t$ is apparent), and the Fano curve parametrized by $[\mathbf{X}_t](\lambda)$ is the elliptic curve $$X_1^3X_2^3+X_2^3X_3^3+X_3^3X_1^3-tX_1^2X_2^2X_3^3=0$$ with j-invariant $$j(1:-t)=t^3\frac{(t^3+216)^3}{(t^3-27)^3},$$ for $t\ne 3,3e^{2\pi i/3},3e^{-2\pi i/3}$. The discriminant of the denominator $(t^3+9\lambda t+54)(\lambda^2 t^2+4\lambda^3-4t^3-18\lambda t-27)$ is $16(5t^3-864)^2(t^3-27)^7$. As before, in the cases $t=3,3e^{2\pi i/3},3e^{-2\pi i/3}$ the equation factors. In the case when $5t^3-864=0$, the apparent singularity and a root of $\lambda^2 t^2+4\lambda^3-4t^3-18\lambda t-27$ are confluent singularities.

\section{Example: Fricke degree $12$ pencil}

In this section, let $G\subset SL_3(\mathbb{C})$ be the isomorphic representation of Klein's simple group of order $168$ generated by the matrices in Example \ref{Klein} and $P_4$, $P_6$ and $P_{14}$ the three algebraically independent $G$-invariant polynomials in $\mathbb{C}[X_1,X_2,X_3]$.

Fricke degree $12$ pencil is the collection of $G$-invariant curves of degree $12$ in $\mathbb{P}^2(\mathbb{C})$
$$C_{12}(\nu:\mu)=\left\{(X_1:X_2:X_3)\in\mathbb{P}^2(\mathbb{C})\Big|\ (\nu P^2_6+\mu P^3_4)(X_1:X_2:X_3)=0\right\}.$$
We have the generic case $C_{12}(1:\mu)$ that has genus $31$ if $\mu\ne 0,-\frac{4}{27},4$, and the singular cases $C_{12}(1:-\frac{4}{27})$ with genus $3$, $C_{12}(1:4)$ and $C_{12}(1:0)$ with genus $10$ and Klein's quartic $C_{12}(0:1)$ of genus $3$.

Kato in \cite{KATO2004} produced linear ordinary differential equations whose associated Schwarz maps parametrizes $C_{12}(1:\mu)$, $C_{12}(1:-\frac{4}{27})$ and $C_{12}(1:4)$.

\begin{exa}\label{KatoI}
The Schwarz map associated to a full system of solutions to the linear ordinary differential equation
\begin{eqnarray*}
0 & = & \left(\frac{d}{dt}\right)^3x+\frac{1}{2}\frac{7t-4}{t(t-1)}\left(\frac{d}{dt}\right)^2x+\frac{1}{252}\frac{387t-56}{t^2(t-1)}\left(\frac{d}{dt}\right)x \\
 & & \ -\frac{15}{2744}\frac{1}{t^2(t-1)}x
\end{eqnarray*}
parametrizes $C_{12}(1:-\frac{4}{27})$. In particular we can choose solutions $x_1,x_2,x_3$ such that
\begin{eqnarray*}
\frac{-3P_{14}+8P_4^2P_6}{288P_4^2P_6}(x_1,x_2,x_3)& = & t
\end{eqnarray*}
and therefore the equation is standard. Note that by definition of $C_{12}(1:-\frac{4}{27})$ we have
\begin{eqnarray*}
 P_6^2-\frac{4}{27}P_4^3(x_1,x_2,x_3)& = & 0.
\end{eqnarray*} 
\end{exa}
 
\begin{exa}\label{KatoII}
The Schwarz map associated to a full system of solutions to the linear ordinary differential equation
\begin{eqnarray*}
0 & = & \left(\frac{d}{dt}\right)^3x+\frac{1}{2}\frac{7t-4}{t(t-1)}\left(\frac{d}{dt}\right)^2x+\frac{1}{112}\frac{172t-21}{t^2(t-1)}\left(\frac{d}{dt}\right)x \\
 & & \ -\frac{15}{2744}\frac{1}{t^2(t-1)}x
\end{eqnarray*}
parametrizes $C_{12}(1:4)$. As in the previous example we can choose solutions $x_1,x_2,x_3$ such that
\begin{eqnarray*}
\frac{P_{14}+72P_4^2P_6}{128P_4^2P_6}(x_1,x_2,x_3)& = & t
\end{eqnarray*}
and therefore the equation is standard. Note that by definition of $C_{12}(1:4)$ we have
\begin{eqnarray*}
 P_6^2+4P_4^3(x_1,x_2,x_3)& = & 0.
\end{eqnarray*} 
\end{exa}

\begin{exa}\label{KatoIII}
The Schwarz map associated to a full system of solutions to the linear ordinary differential equation
\begin{eqnarray*}
0 & = & \left(\frac{d}{dt}\right)^3x+\left(\frac{3}{2}\frac{\partial}{\partial t}\ln P(\mu,t)-\frac{1}{t-t_0}\right)\left(\frac{d}{dt}\right)^2x\\
 & &\ +\left(\frac{43}{28}t+\frac{729\mu^2-3628\mu-640}{21\mu}\right.\\
 & &\quad \left.-\frac{(2187\mu^2-15004\mu-3200)(27\mu+4)(\mu-4)}{63\mu^2(t-t_0)}\right)\frac{1}{P(\mu,t)}\left(\frac{d}{dt}\right)x \\
 & &\qquad +\left(-\frac{15}{14^3}-\frac{5(27\mu+4)(\mu-4)}{196\mu(t-t_0)}\right)\frac{1}{P(\mu,t)}x,
\end{eqnarray*}
where $$t_0=-\frac{81\mu^2-432\mu-80}{3\mu}$$ and
$$P(\mu,t)  =  \frac{1}{4}\left(4t^3-g_2(\mu)t-g_3(\mu)\right),$$
with
$$g_2(\mu)  =  \frac{64}{3\mu}\left(189\mu^2+280\mu-48\right)$$
and
$$g_3(\mu)  =  -\frac{-256}{27\mu}\left(729\mu^3+12852\mu^2+1456\mu+4032\right)$$
parametrizes $C_{12}(1:\mu)$, for $\mu\ne 0,-\frac{4}{27},4$. We can choose solutions $x_1,x_2,x_3$ such that
\begin{eqnarray*}
\frac{-3P_{14}-88P_4^2P_6}{3P_4^2P_6}(x_1,x_2,x_3)& = & t
\end{eqnarray*}
and therefore the equation is standard. Note that by definition of $C_{12}(1:\mu)$ we have
\begin{eqnarray*}
 P_6^2+\mu P_4^3(x_1,x_2,x_3)& = & 0.
\end{eqnarray*} 
\end{exa}

The Schwarz maps associated to the equations in Examples \ref{KatoI}, \ref{KatoII} and \ref{KatoIII} correspond to different orbits $\mathbf{X}_{\mu}(t)=(x_{\mu,1},x_{\mu,2},x_{\mu,3})$ of a single $G$-homogeneous dynamical system of degree $18$ defined by the equations
\begin{eqnarray*}
P_6^3(x_{\mu,1},x_{\mu,2},x_{\mu,3}) & = & -\mu \\
P_{14}P_4(x_{\mu,1},x_{\mu,2},x_{\mu,3}) & = & 2\left((27\mu-44)t-9\mu\right) \\
P_4^3P_6(x_{\mu,1},x_{\mu,2},x_{\mu,3}) & = & 1.
\end{eqnarray*}
If $\mu\ne 0,44/27$, the curve parametrized by $\mathbf{X}_{\mu}(t)$ is $C_{12}(1:\mu)$, for $$(P_6^3-\mu P_4^3P_6) (x_{\mu,1},x_{\mu,2},x_{\mu,3})=P_6(P_6^2-\mu P_4^3) (x_{\mu,1},x_{\mu,2},x_{\mu,3})=0.$$
Now, when $\mu=-\frac{4}{27}$ and $\mu=4$ we obtain the same equations as in Examples \ref{KatoI} and \ref{KatoII}. But the generic case becomes a projectively equivalent equation
\begin{eqnarray*}
0 & = & \left(\frac{d}{dt}\right)^3x+\left(\frac{3}{2}\frac{\partial}{\partial t}\ln Q(\mu,t)-\frac{1}{t-\tau_0}\right)\left(\frac{d}{dt}\right)^2x\\
 & &\ +\left(\frac{43(27\mu-44)^2\mu}{551124}t-\frac{(2619\mu^2-9148\mu-1280)(27\mu-44)}{1653372}\right.\\
 & &\quad \left.-\frac{(27\mu+4)(\mu-4)(2187\mu^2-15004\mu-3200)}{4960116\mu(t-\tau_0)}\right)\frac{1}{Q(\mu,t)}\left(\frac{d}{dt}\right)x \\
 & &\qquad +\left(-\frac{5(27\mu-44)^2\mu}{18003384}-\frac{5(27\mu+4)(\mu-4)(27\mu-44)}{7715736(t-\tau_0)}\right)\frac{1}{Q(\mu,t)}x
\end{eqnarray*}
where $$\tau_0=\frac{5}{6}\frac{(27\mu+4)(\mu-4)}{\mu(27\mu-44)}$$ and
\begin{eqnarray*}
Q(\mu,t) & = &
   \frac{1}{19683}\left(h_0(\mu)t^3+h_1(\mu)t^2+h_2(\mu)t+h_3(\mu)\right); \textrm{ with }\\
h_0(\mu) & = & (27\mu-44)^2\mu,\\
h_1(\mu) & = & -(27\mu-44)^2\mu,\\
h_2(\mu) & = & (9\mu-4)(27\mu+4)(\mu-4),\textrm{ and }\\
g_3(\mu) & = & -(27\mu+4)(\mu-4)^2.
\end{eqnarray*}
With this unified viewpoint we can study the family of Schwarz maps $[\mathbf{X}_\mu](t)$ in terms of the deformation parameter $\mu$. Let us denote $$\mathbf{X}_t(\mu)=\mathbf{X}(\mu,t)=\mathbf{X}_{\mu}(t)$$
and $x_{t,i}(\mu)=x_i(\mu,t)=x_{\mu,i}(t)$, for $i=1,2,3$. So we have
\begin{eqnarray*}
\frac{d}{d\mu}P_6^3(x_{t,1},x_{t,2},x_{t,3}) & = & -1 \\
\frac{d}{d\mu}P_{14}P_4(x_{t,1},x_{t,2},x_{t,3}) & = & 2(27t-9) \\
\frac{d}{d\mu}P_4^3P_6(x_{t,1},x_{t,2},x_{t,3}) & = & 0
\end{eqnarray*}
and a linear differential equation of the form
\begin{eqnarray*}
0 & = & \left(\frac{d}{d\mu}\right)^3x+\frac{1}{\mu}\frac{P_2(t,\mu)}{\Delta(t,\mu)}\left(\frac{d}{d\mu}\right)^2x+\frac{1}{\mu^2}\frac{P_1(t,\mu)}{\Delta(t,\mu)}\left(\frac{d}{d\mu}\right)x+\frac{1}{\mu^3}\frac{P_0(t,\mu)}{\Delta(t,\mu)}x
\end{eqnarray*}
where $\Delta(t,\mu)$, $P_2(t,\mu)$, $P_1(t,\mu)$, and $P_0(t,\mu)$ are polynomials of degree $9$ in $\mu$ and $7$ in $t$. If we fix $t=0$ we obtain the equation
{\scriptsize
\begin{eqnarray*}
0 & = & \left(\frac{d}{d\mu}\right)^3x+\frac{2}{3}\frac{185895\mu^5-3763773\mu^4+5350808\mu^3+4783200\mu^2+41088\mu-19712}{(27\mu+4)(\mu-4)(1377\mu^3-19814\mu^2-8720\mu+352)\mu}\left(\frac{d}{d\mu}\right)^2x\\
 & & \quad +\frac{1}{189}\frac{1}{(27\mu+4)^2(\mu-4)^2(1377\mu^3-19814\mu^2-8720\mu+352)\mu^2}(250958250\mu^7-7572554568\mu^6\\
 & &\qquad +22739097249\mu^5+11647351846\mu^4-19225904752\mu^3-350929536\mu^2+350506496\mu\\
 & &\ \qquad +7135744)\left(\frac{d}{d\mu}\right)x\\
 & & \quad-\frac{17}{250047}\frac{1}{(27\mu+4)^2(\mu-4)^2(1377\mu^3-19814\mu^2-8720\mu+352)\mu^3}(25981560\mu^7-4643083377\mu^6\\
& & \qquad +5059137096\mu^5+27082814356\mu^4+12246977888\mu^3+4486312704\mu^2+569635328\mu-3863552)x
\end{eqnarray*}}that has singularities at the values $\mu=-4/27,0,4$, and at infinity. The zeroes of $1377\mu^3-19814\mu^2-8720\mu+352$ are apparent singularities.

\bibliographystyle{plain}
\bibliography{DGT}

\end{document}